\documentclass[11pt,leqno]{amsart}

\usepackage[utf8x]{inputenc}
\usepackage[T1]{fontenc}
\usepackage[english]{babel}

\usepackage[margin=1.1in]{geometry}

\usepackage{amsmath,amssymb,amsthm}
\usepackage{mathabx}
\usepackage{yfonts}
\usepackage{mathrsfs}
\usepackage{faktor}
\usepackage[lite,initials, alphabetic]{amsrefs}
\usepackage[dvipsnames]{xcolor}
\usepackage[colorlinks=true, linkcolor=MidnightBlue, citecolor=OliveGreen]{hyperref}
\usepackage{comment}

\usepackage{graphicx}
\usepackage{tikz}
\usepackage{tikz-cd}
\usetikzlibrary{arrows, automata, positioning}

\theoremstyle{definition}
\newtheorem{defn}{{\bf Definition}}[section]

\newtheorem{rmk}[defn]{{\bf Remark}}
\newtheorem{question}[defn]{{\bf Question}}
\theoremstyle{theorem}
\newtheorem{lemma}[defn]{{\bf Lemma}}
\newtheorem{theorem}[defn]{{\bf Theorem}}
\newtheorem{coro}[defn]{{\bf Corollary}}
\newtheorem{prop}[defn]{{\bf Proposition}}

\DeclareMathOperator{\Sym}{\mathrm{Sym}}
\DeclareMathOperator{\Aut}{Aut}

\DeclareMathOperator{\st}{st}
\DeclareMathOperator{\St}{St}
\DeclareMathOperator{\rst}{rist}
\DeclareMathOperator{\Rst}{Rist}

\newcommand{\LD}{\langle}
\newcommand{\RD}{\rangle}

\numberwithin{equation}{section}

\DeclareMathOperator{\bp}{Bas}
\newcommand{\cref}[2]{\hyperref[#2]{#1~\ref*{#2}}}

\title{The Basilica group of the Grigorchuk--Erschler group}

\author[K. Rajeev]{Karthika Rajeev} 
\address{Karthika Rajeev: Fakult\"at f\"ur Mathematik,
Universit\"at Bielefeld,
33501 Bielefeld, Germany}
\email{krajeev@math.uni-bielefeld.de}

\author[A. Thillaisundaram]{Anitha Thillaisundaram} 
\address{Anitha Thillaisundaram: Centre for Mathematical Sciences, Lund University,  
223 62 Lund, Sweden}
\email{anitha.thillaisundaram@math.lu.se}

\date{\today}

\thanks{This research was supported by the Deutsche Forschungsgemeinschaft (DFG, German Research Foundation) – Project-ID 491392403 – TRR 358  and by the GS Magnuson Foundation, grant MG2023-0043. The second author thanks Bielefeld University for their hospitality.}



\begin{document}

\maketitle

\begin{abstract}
Francoeur and Garrido recently provided the first explicit examples of finitely generated branch groups with maximal subgroups of infinite index, where one of their examples is the Grigorchuk--Erschler group. Here, using the Basilica operation constructed by Petschick and Rajeev, we show that the second Basilica group of the Grigorchuk--Erschler group also has maximal subgroups of infinite index.
\end{abstract}

\section{Introduction}

Branch groups form a  well-studied subfamily of groups acting on rooted trees, due to the tighter structure and useful properties that branch groups satisfy. Roughly speaking, a branch group has a subnormal subgroup structure that is similar to the full automorphism group of the rooted tree. From the first time branch groups were considered, they have provided explicit examples of groups with interesting properties, that have in turn been used to answer several important questions, such as the Burnside Problem, the Mahlon Day Problem, to name but a few. They have also shown to play a key role in the theory of infinite groups, specifically for just infinite groups, and in recent times one has seen connections between  branch groups and other areas of mathematics, such as dynamics, algebraic geometry, and cryptography. 

One particular aspect of branch groups that has attracted attention is the study of their maximal subgroups. Prompted by the work of Pervova~\cite{Per00}, who showed that some early examples of finitely generated branch groups only had maximal subgroups of finite index, Grigorchuk asked about the existence of finitely generated branch groups with maximal subgroups of infinite index. This was answered positively by Bondarenko~\cite{Bondarenko} in 2010, though his construction was not explicit. The first explicit examples of finitely generated branch groups with maximal subgroups of infinite index were provided by Francoeur and Garrido~\cite{FG}, who showed that the non-torsion \u{S}uni\'{c} groups acting on the binary rooted tree have maximal subgroups of infinite index. Until recently, these remained the only examples of such groups. A new and interesting class of examples have been provided by Garciarena and Petschick (see \cite{GP}), which are of a different flavour to the  \u{S}uni\'{c} groups, and thus also exhibit different and interesting properties, such as being torsion. The techniques in \cite{GP} do not apply to branch groups with generators where their definition varies depending on the level of the tree, such as the \u{S}uni\'{c} groups. In this paper, we provide another explicit example (as part of a more general approach)  of a finitely generated branch group with maximal subgroups of infinite index, that is similar to the \u{S}uni\'{c} groups.

We will exclusively be interested in branch groups acting on the binary rooted tree, which we denote by $T$. The Grigorchuk--Erschler group $G$ is a $3$-generated group acting on $T$. Writing $a$ for the automorphism of $T$ that swaps the two maximal subtrees, the generators of $G$ are $a$ together with $b_0$ and $b_1$, the latter two which are defined recursively as follows:
\begin{align*}
   b_0= (b_1,1), && b_1=(b_0,a),
\end{align*}
where  $(x,y)$ represents the independent action on the two maximal subtrees, where $x,y\in \Aut T$. The Grigorchuk--Erschler group is one out of the family of Grigorchuk groups of intermediate growth, introduced in~\cite{Grigorchuk}, and particularly good bounds were given by Erschler~\cite{Erschler} for the growth function. The Grigorchuk-Erschler group is also the only self-similar group that is both in the family of Grigorchuk groups from~\cite{Grigorchuk} and in the family of \u{S}uni\'{c} groups~\cite{Sunic}. The Grigorchuk-Erschler group has also featured in some recent work concerning iterated monodromy groups~\cite{Nekrashevych}.

Petschick and Rajeev~\cite{PR} introduced a general construction, inspired by  the Basilica group which was in turn introduced by Grigorchuk and \.{Z}uk in \cite{GZ01} and \cite{GZ02}. The construction in~\cite{PR} works for more general regular rooted trees, but here we recall it only for the binary rooted tree. Let $s \geq 2$ be an integer and let $G$ be a subgroup of the automorphism group $\Aut T$ of the binary rooted tree~$T$. The $s$th Basilica group of $G$ is given by $\bp_{s}(G) = \LD\{\beta^s_i(g)\mid g \in G, \, i \in \{0,1,\dots, s-1\}\}\,\RD$, where $\beta_i^s: \Aut T \rightarrow \Aut T$ are monomorphisms given by
\begin{align*}
    \beta^s_i(g) &= (\beta^s_{i-1}(g), 1) &\text{ for } i \in \{1,\dots,s-1\},\\
		\beta^s_0(g) &=  (\beta^s_{s-1}(g|_0), \beta^s_{s-1}(g|_{1}))g|^\epsilon,
\end{align*}
 where $g|_x$ is the restriction of~$g$ to the subtree rooted at a first-level vertex $x$, and  $g|^\epsilon$ is the local action of the element~$g$ at the root of~$T$. That is, since $T$ is the binary rooted tree, here $g|^\epsilon\in C_2$. For more information on the Basilica operation, we refer the reader to~\cite{PR}.
 
In this paper, we use the Basilica operation on the Grigorchuk--Erschler group to produce a new explicit example of a finitely generated branch group with maximal subgroups of infinite index.

\begin{theorem}
    \label{thm:main} The second Basilica group of the Grigorchuk-Erschler group has maximal subgroups of infinite index.
\end{theorem}

\noindent We comment that the theorem should also hold for the $s$th Basilica group of the Grigorchuk-Erschler group, but due to the notational complexity, we only consider the case $s=2$ here.

Our proof of the above result uses the well-established approach by Pervova~\cite{Per00}, which was developed by Francoeur~\cite{Fra20}, which is to establish the existence of proper subgroups that are dense with respect to the profinite topology. The end of our proof is analogous to the strategy taken in~\cite{FG}, which is to consider the orbital graph of the action of the group on the boundary of the binary rooted tree.

We also establish some basic properties of $B$, such as having the congruence subgroup property and being branch over its derived subgroup; cf.  Proposition~\ref{prop:csp}   and Lemma~\ref{lemma branch} respectively. We further compute the Hausdorff dimension of the closure of $B$, and we prove that $\Aut B=N_{\Aut T}(B)$; see Proposition~\ref{prop:hausdorff-dimension} and Theorem~\ref{thm:rigidity} respectively.

As a side result, we also make the following observation, which holds more generally for groups acting on any regular rooted tree. Note that the Basilica group of a branch group is weakly branch; cf. \cite[Thm.~1.1]{PR}.

\begin{theorem}\label{thm:Basilica-dense}
Let $G$ be a branch group with the congruence subgroup property, and suppose for some $s\ge 2$ that $\bp_{s}(G)$ also has the congruence subgroup property. If $H$ is dense in $G$, then $\bp_{s}(H)$ is dense in $\bp_{s}(G)$.
\end{theorem}

Referring to Theorem \ref{thm:main}, it would be interesting to see if the property of having maximal subgroups of infinite index is always preserved by the Basilica operation. That is:
\begin{question}
    Let $G$ be a finitely generated branch group with maximal subgroups of infinite index. Does $\bp_{s}(G)$, the $s$th Basilica group of $G$, for any $s\ge 2$, still have maximal subgroups of infinite index?
\end{question}

\smallskip

\noindent\textit{Organisation}. In Section~\ref{sec:prelim} we introduce basic definitions and necessary terminology for groups acting on the binary rooted tree, and here we also give the definition of the second Basilica group~$B$ of the Grigorchuk--Erschler group. We next establish some basic properties of the group~$B$ in Section~\ref{sec:properties}. In Section~\ref{sec:operation-dense}, we prove Theorem~\ref{thm:Basilica-dense}. Finally, in Section~\ref{sec:maximal}, we prove Theorem~\ref{thm:main}.

\medskip

\noindent\textit{Notation.}
Throughout, we  use left-normed commutators, for example, $[x,y,z] = [[x,y],z]$.

\section{Preliminaries}\label{sec:prelim}

\subsection{Groups acting on the binary rooted tree}
Let $X = \{0, 1\}$ and $X^*$ denote the set of all finite words over the alphabet~$X$. We naturally identify~$X^*$ with the binary rooted tree~$T$, where the root corresponds to the empty word $\epsilon$, and each word~$v \in X^*$ represents a vertex. A vertex~$v$ is connected by an edge to its children,~$v0$ and~$v1$, which are obtained by appending~$0$ and~$1$ to~$v$, respectively. The levels of~$T$ correspond to the lengths of the words in~$X^*$. The words of length~$n$ in~$X^*$ form the $n$th layer of~$T$ and is denoted by~$X^n$. In the sequel, we do not distinguish between the words in~$X^*$ and the vertices of~$T$.

An automorphism of $T$ is a graph automorphism that fixes the root and preserves the adjacency of vertices. Such an automorphism permutes the vertices within each level of the tree while maintaining the tree structure. The set of all automorphisms of $T$ forms a group under composition, denoted by $\Aut T$.

For every $v \in T$, let $T_v$ denote the subtree of $T$ rooted at $v$, consisting of the set of words $\{vw \mid w \in X^*\}$. Each subtree $T_v$ can be identified with the original tree $T$ by mapping a vertex $vw \in T_v$ to the vertex $w \in T$ for every $w \in X^*$.  

Now, let $g \in \Aut T$. Identifying both $T_v$ and $T_{g(v)}$ with the original tree $T$, we obtain an automorphism of $T$, \emph{the section of $g$ at $v$}, denoted by $g|_v$, which is uniquely determined by the rule:
\[
g(vw) = g(v)g|_v(w), \quad \text{for all } w \in X^*.
\]
By following this convention, for every $n \in \mathbb{N}$, the action of $g$ can be uniquely described as  
\[
g = (g|_{x_1}, \dots, g|_{x_{2^n}}) \, \sigma,
\]
where $g|_{x_i}$ denotes the section of $g$ at $x_i \in X^n$, with $x_1, \dots, x_{2^n}$ arranged lexicographically. The term $\sigma \in \Sym X^n$ represents the action of $g$ restricted to the $n$th level of the tree $T$.

\subsection{Key subgroups of $\Aut T$}

Let $G \leq \Aut T$. The subgroup $G \leq \Aut T$ is called \emph{self-similar} if, for every $v \in T$, the set $G_v = \{g|_v \, | \, g \in G\}$ is contained in $G$, and is called \emph{spherically transitive} if $G$ acts transitively on each layer of $T$.

A \emph{vertex stabiliser} $\st_G(v)$ of a vertex $v \in T$ in $G$ is the set of all $g \in G$ that fix $v$, which is clearly a group under composition. The intersection of all vertex stabilisers of the vertices at level $n$ forms the \emph{$n$th level stabiliser} in $G$, denoted by $\St_G(n)$. The $n$th level stabiliser can be seen as the kernel of the induced action of $G$ on $X^n$. Hence, $\St_G(n)$ is a normal subgroup of $G$ and has finite index in $G$. Analogous to the classical congruence subgroup property for $\text{SL}_n(\mathbb{Z})$, we say $G$ has the \emph{congruence subgroup property} if every subgroup of finite index in $G$ contains a level stabiliser $\St_G(n)$ for some $n \in \mathbb{N}$. 

The elements of $\St_G(n)$ can be completely determined by the following map
\begin{align*}
\psi_n : \St_G(n) &\rightarrow \Aut T \times \overset{2^n}\ldots \times \Aut T\\
 g & \mapsto (g|_{x_1}, \dots, g|_{x_{2^n}}).
\end{align*}
If $G$ is self-similar then $\St_G(n)$ can be identified with a subgroup of $G \times \overset{2^n}\ldots \times G$. In that case, as done at the end of the previous subsection, we identify an element $\St_G(n)$ with its image under the monomorphism $\psi_n$. 

A spherically transitive subgroup $G$ is called \emph{fractal} if $(\st_G(v))_v = G$ for every $v \in T$, and is called \emph{strongly fractal} if $(\St_G(1))_x =G$ for every $x \in X$.

A \emph{rigid vertex stabiliser} $\rst_G(v)$ for $v\in T$ is the subgroup of all elements in $G$ that fix every vertex outside the subtree $T_v$. For $v \in X^n$ and $w \in X^m$ with $n \leq m$, we either have $\rst_G(w) \leq \rst_G(v)$, if~$v$ is a prefix of~$w$, or $\rst_G(v) \cap \rst_G(w) = \emptyset$. The \emph{$n$th level rigid stabiliser} $\Rst_G(n)$ is the subgroup generated by the union of all rigid vertex stabilisers of the vertices at level $n$:
\[
\Rst_G(n) =  \langle \rst_G(v)\mid  \, v\in X^n \rangle = \prod \limits_{v \in X^n} \rst_G(v).
\]

Let $G \leq \Aut T$ be spherically transitive. The group $G$ is said to be \emph{weakly branch} if all $\Rst_G(n)$ are non-trivial. Further~$G$ is said to be \emph{branch} if, in addition, $\Rst_G(n)$ has finite index in $G$ for every $n$. 

Now assume that $G \leq \Aut T$ is self-similar and spherically transitive. If there exists a non-trivial subgroup $K$ in~$G$ such that $\psi_1^{-1}(K \times K)$ is contained in~$K$ then~$G$ is said to be \emph{weakly regular branch} over~$K$. If~$K$ also has finite index in~$G$ then~$G$ is said to be \emph{regular branch} over~$K$.

We record a straightforward fact for later use.
\begin{lemma}\label{lem:derived-contained-st}
Let $H$ be a  subgroup of $\Aut T$. Then the quotient $\St_{H}(n-1)/\St_H(n)$ is elementary abelian for all $n\in\mathbb{N}$. As a consequence $\St_H(n-1)'\le \St_H(n)$ for all $n\in\mathbb{N}$.
\end{lemma}

\begin{proof}
Indeed, first notice that $\St_{\Aut T}(n-1)/\St_{\Aut T}(n)\cong C_2 \times \overset{2^{n-1}}\dots \times C_2$. Since $\St_{H}(n)=\St_{\Aut T}(n) \cap H$ it follows that $\St_H(n-1)/\St_H(n)$ embeds in $\St_{\Aut T}(n-1)/\St_{\Aut T}(n)$, and consequently $\St_{H}(n-1)/\St_H(n)$ is elementary abelian.
\end{proof}

\subsection{The Grigorchuk--Erschler group and the Basilica operation}

We recall that the Grigorchuk--Erschler group $G$ is a $3$-generated group acting on the binary rooted tree. Writing $\tau$ for the non-trivial transposition in $\text{Sym}(\{0,1\})$, the generators $a$, $b_0$, $b_1$ of $G$ are defined as follows:
\begin{align*}
    a = \tau, 
    && b_0= (b_1,1), && b_1=(b_0,a).
\end{align*}
As indicated before, for notational convenience, in the sequel we will often omit writing $\psi_1$, or more generally  $\psi_n$  for $n\in\mathbb{N}$.

We recall the Basilica operation from \cite{PR} for groups acting on the binary rooted tree:
Let $s \geq 2$ be an integer and let $G$ be a subgroup of the automorphism group $\Aut T$. The $s$th Basilica group of $G$ is given by $\bp_s(G) = \LD\{\beta^s_i(g)\mid g \in G, \, i \in \{0,1,\dots, s-1\}\}\,\RD$, where $\beta_i^s: \Aut T \rightarrow \Aut T$ are monomorphisms given by
\begin{align*}
    \beta^s_i(g) &= (\beta^s_{i-1}(g),1) &\text{ for } i \in \{1,\dots,s-1\},\\
		\beta^s_0(g) &=  (\beta^s_{s-1}(g|_0),  \beta^s_{s-1}(g|_1))g|^\epsilon,
\end{align*}
 where $g|_x$ is the restriction of~$g$ to the subtree rooted at a first-level vertex $x$, and  $g|^\epsilon$ is the local action of the element~$g$ at the root of~$T$. Hence the second Basilica group $B:=\bp_2(G)$ of $G$ is given by
\[B = \langle \beta_0^2(a), \beta_0^2(b_0),\beta_0^2(b_1),\beta_1^2(a), \beta_1^2(b_0),\beta_1^2(b_1)\rangle.\]
For convenience, we omit the superscript $2$ from the $\beta_i$ maps. We have
\begin{align*}
    \beta_0(a) &= (1,1)\tau = a, && \beta_1(a) = (\beta_0(a),1) = (a,1), \\
    \beta_0(b_0) &= (\beta_1(b_1),1), && \beta_1(b_0) = (\beta_0(b_0),1),\\
    \beta_0(b_1) &= (\beta_1(b_0),\beta_1(a)), && \beta_1(b_1) = (\beta_0(b_1),1).
\end{align*}
Now set $\alpha = \beta_1(a), c_0 = \beta_1(b_0), c_1 = \beta_0(b_0), c_2 = \beta_1(b_1)$ and $c_3 = \beta_0(b_1)$. Then the second Basilica group $B$ of $G$ is generated by the elements
\begin{align*}
    a, \quad\alpha = (a,1), \quad c_0 = (c_1,1), \quad c_1 = (c_2,1), \quad c_2 = (c_3,1), \quad c_3 = (c_0,\alpha).
\end{align*}

\section{Properties of $B$}\label{sec:properties}

The following result is immediate from \cite[Thm.~1.1]{PR}.
\begin{lemma}
 Let $B$ be the second Basilica of the Grigorchuk--Erschler group. Then $B$ is
 \begin{enumerate}
		\item [(i)] spherically transitive,
		\item [(ii)] self-similar, 
		\item [(iii)] strongly fractal, 
		\item [(iv)] weakly branch.
	\end{enumerate}
\end{lemma}

\begin{lemma}\label{lem:relations}
    Let $B$ be the second Basilica of the Grigorchuk--Erschler group. Then, for all $i,j \in \{0,1,2,3\}, k,\ell \in \{0,1,2\}$ and $s,t \in \{0,1,3\}$, the following equalities hold: 
    \begin{enumerate}
       \item[(i)] $a^2 = 1, \alpha^2 = 1, c_i^2 = 1,$
        \item[(ii)] $[c_i,c_j] = 1,$ 
        \item[(iii)] $[c_k,c_{\ell}^{a}] = 1,$  
        \item[(iv)] $[c_s,c_t^{\alpha}] = 1,$  
        \item[(v)] $[c_k,a,\alpha] = [c_k,\alpha].$
    \end{enumerate}
\end{lemma}
\begin{proof}
    The proofs of the statements are straightforward from the definition of  the generators of~$B$.
\end{proof}

\subsection{Key subgroups of $B$}

The following lemma is in contrast with the fact that the Grigorchuk-Erschler group is not regular branch over its derived subgroup.

\begin{lemma} \label{lemma branch}
Let $B$ be the second Basilica of the Grigorchuk--Erschler group. Then $B$ is regular branch over its commutator subgroup.
\end{lemma}

\begin{proof} From Lemma~\ref{lem:relations}(ii) the elements $c_i$ and $c_j$ commute for all $i,j \in \{0,1,2,3\}$. Furthermore,
\begin{align*}
    [\alpha,c_3^{a}] &= ([a,\alpha],1),\\
    [c_0,\alpha]  &= ([c_1,a],1), && [c_1,\alpha]  = ([c_2,a],1), \\
    [c_2,\alpha]  &= ([c_3,a],1), && [c_3,\alpha]  = ([c_0,a],1), \\
    [c_0,c_3^{a}] &= ([c_1,\alpha],1), && [c_1,c_3^{a}] = ([c_2,\alpha],1),\\
    [c_2,c_3^{a}] &= ([c_3,\alpha],1), && [c_3,\alpha,c_3^{a}] = ([c_0,a,\alpha],1) = ([c_0,\alpha],1).
\end{align*}
So $B'\times B'\le B'$. As $B'$ clearly has finite index in $B$, since $B$ is generated by 6 elements of order~2, the result follows.
\end{proof}

\begin{prop}\label{prop:stabilisers}
    Let $B$ be the second Basilica of the Grigorchuk--Erschler group. Then
    \begin{align*}
    \St_B(1)&=\langle \alpha, c_3,c_2,c_1,c_0\rangle^B\\
    \St_B(2)&=\langle  c_3,c_2,c_1,c_0\rangle^B\\
    \St_B(3)&=\langle  
   [c_3,\alpha],c_2,c_1,c_0\rangle^B\\
     \St_B(4)&=\langle  
     [c_3,\alpha], [c_2,a,c_3],c_1,c_0\rangle^B\\
      \St_B(5)&=\langle  
      [c_3,\alpha],[c_2,a,c_3],[c_1,\alpha,c_2],c_0\rangle^B\\
       \St_B(6)&=\langle  
     [c_3,\alpha],[c_2,a,c_3],[c_1,\alpha,c_2],[c_0,c_3^a,c_2]\rangle^B
    \end{align*}
  and 
\begin{align*}
\log_2|B:\St_B(1)|&=1\\
\log_2|B:\St_B(2)|&=3\\
\log_2|B:\St_B(3)|&=7\\
\log_2|B:\St_B(4)|&=15\\
\log_2|B:\St_B(5)|&=31\\
    \log_2|B:\St_B(6)|&=63.
\end{align*}
\end{prop}

\begin{proof}
Note that
\begin{align*}
    a\in \St_B(0)\backslash \St_B(1),\\
     \alpha\in \St_B(1)\backslash \St_B(2),\\
      c_3\in \St_B(2)\backslash \St_B(3),\\
     c_2\in \St_B(3)\backslash \St_B(4),\\
     c_1\in \St_B(4)\backslash \St_B(5),\\
     c_0\in \St_B(5)\backslash \St_B(6).
\end{align*}
The first level stabiliser is straightforward. For the others, we first show by a direct computation below, that the right-hand sides are contained in the respective left-hand sides.  Indeed:
\begin{align*}
   [c_3,\alpha]&=([c_0,a],1)=(c_1,c_1,1,1)\in\St_B(6)\backslash \St_B(7)\\
   [c_2,a,c_3]&=(1,[c_3,\alpha])=(1,1,[c_0,a],1)\in \St_B(7)\backslash \St_B(8)\\
[c_1,\alpha,c_2]&=([c_2,a,c_3],1)\in \St_B(8)\backslash \St_B(9)\\
[c_0,c_3^a,c_1]&=([c_1,\alpha,c_2],1)\in \St_B(9)\backslash \St_B(10) 
\end{align*} 

We now show equality, by considering the quotient of the left-hand side by the right-hand side. Our strategy is to give a minimal generating set for each quotient $\St_B(i)/\St_B(i+1)$, which gives the second part of the statement. To obtain a minimal generating set for $\St_B(i)/\St_B(i+1)$, we recursively use the minimal generating set obtained for $\St_B(i-1)/\St_B(i)$ for describing the sections of generators of $\St_B(i)/\St_B(i+1)$.

For  $\St_B(2)$, we have
\[
\frac{\St_B(2)}{\langle  c_3,c_2,c_1,c_0\rangle^B}\le \frac{\St_B(1)}{\langle  c_3,c_2,c_1,c_0\rangle^B}=\frac{\langle \alpha, c_3,c_2,c_1,c_0\rangle^B}{\langle  c_3,c_2,c_1,c_0\rangle^B}\cong \langle \alpha\rangle^{\langle a,\alpha\rangle}.
\]
Hence, it suffices to show that no non-trivial element in $\langle \alpha\rangle^{\langle a,\alpha\rangle}$ is in $\St_B(2)$. Note that $[a,\alpha]=(a,a)\in\St_B(1)\backslash\St_B(2)$ is in the centre of $\langle a,\alpha\rangle$. Hence 
\[
\langle \alpha\rangle^{\langle a,\alpha\rangle}=\langle \alpha,[a,\alpha]\rangle\cong C_2\times C_2.
\]
A direct check then yields that $\St_B(2)=\langle  c_3,c_2,c_1,c_0\rangle^B$.

For  $\St_B(3)$, we have
\begin{align*}
\frac{\St_B(3)}{\langle  [c_3,\alpha],c_2,c_1,c_0\rangle^B}&\le \frac{\St_B(2)}{\langle  [c_3,\alpha],c_2,c_1,c_0\rangle^B}=\frac{\langle  c_3,c_2,c_1,c_0\rangle^B}{\langle   [c_3,\alpha],c_2,c_1,c_0\rangle^B}.
\end{align*}
Note that $[c_3,a]=(c_0\alpha,\alpha c_0)=(c_1a,1,ac_1,1)$ has order 4. Furthermore 
$[c_3,a,\alpha]=([c_0\alpha,a],1)=(c_1a,ac_1,1,1)\in\St_B(2)\backslash\St_B(3)$  has order 4 and $[c_3,a,\alpha]^2=[c_3,a]^2([c_3,a]^2)^\alpha$. Next, note that
\[
  [c_3,a,\alpha,c_3]=([c_0\alpha,a,c_0],1)=((c_0\alpha)^2,1)\in\St_B(6)\backslash \St_B(7).
  \]
  Further we have that $[c_3,a,\alpha,c_3][c_3,a,\alpha,c_3]^a=[c_3,a]^2=[c_3,a,c_3]=[c_3,a,a]$, and observe that $[c_3,a,\alpha,c_3]=[c_3,\alpha,c_3^a]$. So in particular $$
  [c_3,a,\alpha,c_3],[c_3,a,a],[c_3,a,c_3],
  [c_3,a]^2\in\langle [c_3,\alpha]\rangle^B.
  $$
  Note also that $[c_3,a,\alpha,\alpha]=[c_3,a,\alpha]^2\in \St_B(3)\backslash\St_B(4)$.
  Thus we deduce
\begin{align*}
\frac{\langle  c_3,c_2,c_1,c_0\rangle^B}{\langle    [c_3,\alpha],c_2,c_1,c_0\rangle^B}&\cong \langle c_3\rangle\times \frac{\langle [c_3,a]\rangle}{\langle [c_3,a]^2\rangle}\times \frac{\langle [c_3,a,\alpha]\rangle}{\langle [c_3,a,\alpha]^2\rangle}\times \frac{\langle [c_3,a,\alpha]^a\rangle}{\langle ([c_3,a,\alpha]^2)^a\rangle}\\
&\cong C_2\times C_2\times C_2\times C_2.
\end{align*}
Hence  $\St_B(3)=\langle  [c_3,\alpha],c_2,c_1,c_0\rangle^B$. Indeed, the commutators of the four generators $c_3$, $[c_3,a]$, $[c_3,a,\alpha]$, $[c_3,a,\alpha]^a$ all lie in $\St_B(3)$ since $\St_B(2)'\le \St_{B}(3)$. Also a direct check yields that all possible non-trivial products of these four generators lie in $\St_B(2)\backslash\St_B(3)$.

 For  $\St_B(4)$, we have
\begin{align*}
\frac{\St_B(4)}{\langle  [c_3,\alpha], [c_2,a,c_3],c_1,c_0\rangle^B}&\le \frac{\St_B(3)}{\langle  [c_3,\alpha], [c_2,a,c_3],c_1,c_0 \rangle^B}=\frac{\langle  [c_3,\alpha],c_2,c_1,c_0\rangle^B}{\langle   [c_3,\alpha], [c_2,a,c_3],c_1,c_0\rangle^B}.
\end{align*}
We consider $[c_2,a]=(c_3,c_3)\in \St_B(3)\backslash \St_B(4)$ and $[c_2,\alpha]=([c_3,a],1)\in \St_B(3)\backslash \St_B(4)$. Since $[c_3,a]^2\in\langle [c_3,\alpha]\rangle^B$  as seen above, from
\begin{align*}
    [c_2,a,c_3]&=(1,[c_3,\alpha])\\
[c_2,\alpha]^2&=([c_3,a]^2,1)
\end{align*} 
it follows that $[c_2,\alpha]^2\in\langle  [c_2,a,c_3]\rangle^B$. Also $[c_2,\alpha]^a=(1,[c_3,a])\in \St_B(3)\backslash \St_B(4)$ commutes with $[c_2,\alpha]$. Furthermore  $[c_2,a,a]=[c_2,a,c_2]=1$. Next, by the Hall-Witt identity, we have $[c_2,\alpha,c_3]^{\alpha}=([\alpha,c_3,c_2]^{-1})^{c_3}\in\langle [c_3,\alpha]\rangle^B$. Also recall that $[c_2,a,\alpha]=[c_2,\alpha]$. Lastly,
\begin{align*}
    [c_2,\alpha,c_3^a]&=([c_3,a,\alpha],1)\\
    [c_2,\alpha,c_3^a]^a&=(1,[c_3,a,\alpha])\\
    [c_2,\alpha,c_3^a]^\alpha&=([c_3,a,\alpha]^a,1)\\
     [c_2,\alpha,c_3^a]^{\alpha a}&=(1,[c_3,a,\alpha]^a).
\end{align*}
Thus, and appealing to the description of $\St_B(2)/\St_B(3)$ above, we obtain
\begin{align*}
&\frac{\langle   [c_3,\alpha],c_2,c_1,c_0\rangle^B}{\langle [c_3,\alpha], [c_2,a,c_3],c_1,c_0\rangle^B}\\
&\qquad\cong \langle c_2\rangle\times \langle c_2^a\rangle\times \frac{\langle [c_2,\alpha]\rangle}{\langle [c_2,\alpha]^2\rangle}\times \frac{\langle [c_2,\alpha]^a\rangle}{\langle ([c_2,\alpha]^a)^2\rangle}\\
&\qquad\qquad\qquad\qquad\times \frac{\langle [c_2,\alpha,c_3^a]\rangle}{\langle 
[c_2,\alpha,c_3^a]^2\rangle}\times \frac{\langle [c_2,\alpha,c_3^a]^a\rangle}{\langle ([c_2,\alpha,c_3^a]^a)^2\rangle}\times \frac{\langle [c_2,\alpha,c_3^a]^\alpha\rangle}{\langle ([c_2,\alpha,c_3^a]^\alpha)^2\rangle}\times \frac{\langle [c_2,\alpha,c_3^a]^{\alpha a}\rangle}{\langle ([c_2,\alpha,c_3^a]^{\alpha a})^2\rangle}\\
&\qquad\cong C_2\times \overset{8}\dots\times C_2.
\end{align*}
Indeed, the first-level sections of the above generators yield a minimal generating set for $\St_B(2)/\St_B(3)$, in each component respectively. Hence $\St_B(4)=\langle  [c_3,\alpha], [c_2,a,c_3],c_1,c_0\rangle^B$.

For $\St_B(5)$, we have
\begin{align*}
\frac{\St_B(5)}{\langle   [c_3,\alpha],[c_2,a,c_3],[c_1,\alpha,c_2],c_0\rangle^B}&\le \frac{\St_B(4)}{\langle   [c_3,\alpha],[c_2,a,c_3],[c_1,\alpha,c_2],c_0\rangle^B}\\
&=\frac{\langle   [c_3,\alpha], [c_2,a,c_3],c_1,c_0\rangle^B}{\langle   [c_3,\alpha],[c_2,a,c_3],[c_1,\alpha,c_2],c_0\rangle^B}.
\end{align*}
We consider $[c_1,a]=(c_2,c_2)\in \St_B(4)\backslash \St_B(5)$ and $[c_1,\alpha]=([c_2,a],1)\in \St_B(4)\backslash \St_B(5)$, both of which have order~2. Note that $[c_1,c_3^a]^2=([c_2,\alpha]^2,1)=([c_1,a,c_3]^a)^2$.
Since $[c_2,\alpha]^2\in\langle  [c_2,a,c_3]\rangle^B$,   from
\begin{align*}
    [c_1,a,c_3]^2&=(1,[c_2,\alpha]^2)\\
[c_1,\alpha,c_2]&=([c_2,a,c_3],1)
\end{align*} 
it follows that $ [c_1,a,c_3]^2\in\langle  [c_1,\alpha,c_2]\rangle^B$. Also $[c_1,\alpha]^a=(1,[c_2,a])\in \St_B(4)\backslash \St_B(5)$ commutes with $[c_1,\alpha]$. Furthermore  $[c_1,a,a]=1=[c_1,\alpha,\alpha]$,  and as before by the Hall-Witt identity $[c_1,c_3^a,c_2]=([c_2,\alpha,c_3],1)\in\langle [c_2,a,c_3]\rangle^B$. Recall also that  $[c_1,a,\alpha]=[c_1,\alpha]$.
Thus, upon noting that $[c_1,a,c_3]=(1,[c_2,\alpha])$ and appealing to the description of $\St_B(3)/\St_B(4)$ above, we obtain
\begin{align*}
&\frac{\langle   [c_3,\alpha],[c_2,a,c_3],c_1,c_0\rangle^B}{\langle [c_3,\alpha], [c_2,a,c_3],[c_1,\alpha,c_2],c_0\rangle^B}\\
&\qquad\cong \langle c_1\rangle\times \langle c_1^a\rangle\times \langle [c_1,\alpha]\rangle\times \langle [c_1,\alpha]^a\rangle\\
&\qquad\qquad\times \langle [c_1,a,c_3]\rangle\times \langle [c_1,a,c_3]^a\rangle\times \langle [c_1,a,c_3]^{\alpha^a}\rangle\times \langle [c_1,a,c_3]^{\alpha^a a}\rangle\\
&\qquad\qquad\times \frac{\langle [c_1,c_3^a,c_2^\alpha]\rangle}{\langle [c_1,c_3^a,c_2^\alpha]^2\rangle}\times \frac{\langle [c_1,c_3^a,c_2^\alpha]^a\rangle}{\langle ([c_1,c_3^a,c_2^\alpha]^a)^2\rangle} \times \frac{\langle [c_1,c_3^a,c_2^\alpha]^\alpha\rangle}{\langle ([c_1,c_3^a,c_2^\alpha]^\alpha)^2\rangle}\times \frac{\langle [c_1,c_3^a,c_2^\alpha]^{\alpha a}\rangle}{\langle ([c_1,c_3^a,c_2^\alpha]^{\alpha a})^2\rangle}\\
&\qquad\qquad\times \frac{\langle [c_1,c_3^a,c_2^\alpha]^{c_3^a}\rangle}{\langle ([c_1,c_3^a,c_2^\alpha]^{c_3^a})^2\rangle}\times \frac{\langle [c_1,c_3^a,c_2^\alpha]^{c_3^a a}\rangle}{\langle ([c_1,c_3^a,c_2^\alpha]^{c_3^a a})^2\rangle}\times \frac{\langle [c_1,c_3^a,c_2^\alpha]^{c_3^a\alpha}\rangle}{\langle ([c_1,c_3^a,c_2^\alpha]^{c_3^a\alpha})^2\rangle}\times \frac{\langle [c_1,c_3^a,c_2^\alpha]^{c_3^a \alpha a}\rangle}{\langle ([c_1,c_3^a,c_2^\alpha]^{c_3^a \alpha a})^2\rangle}\\
&\qquad\cong C_2\times \overset{16}\dots\times C_2.
\end{align*}
Hence 
$\St_B(5)=\langle   [c_3,\alpha],[c_2,a,c_3],[c_1,\alpha,c_2],c_0\rangle^B$.

Lastly, for $\St_B(6)$, we have
\begin{align*}
\frac{\St_B(6)}{\langle    [c_3,\alpha],[c_2,a,c_3],[c_1,\alpha,c_2],[c_0,c_3^a,c_2]\rangle^B}&\le \frac{\St_B(5)}{\langle   [c_3,\alpha],[c_2,a,c_3],[c_1,\alpha,c_2],[c_0,c_3^a,c_2]\rangle^B}\\
&=\frac{\langle   [c_3,\alpha],[c_2,a,c_3],[c_1,\alpha,c_2],c_0\rangle^B}{\langle   [c_3,\alpha],[c_2,a,c_3],[c_1,\alpha,c_2],[c_0,c_3^a,c_2]\rangle^B}.
\end{align*}
Likewise note that $[c_0,a]=(c_1,c_1)\in \St_B(5)\backslash \St_B(6)$ and $[c_0,\alpha]=([c_1,a],1)\in \St_B(5)\backslash \St_B(6)$, both of which have order~2. As in the previous case $[c_0,a,a]=1=[c_0,\alpha,\alpha]$  and   $[c_0,a,\alpha]=[c_0,\alpha]$. Note that $[c_0,c_3^a]=([c_1,\alpha],1)=[c_0,a,c_3]^a$ has order 2. Since $[c_1,a,c_3]^2\in\langle  [c_1,\alpha,c_2]\rangle^B$,   we have $ [c_0,\alpha,c_2]^2\in\langle  [c_0,c_3^a,c_3]\rangle^B$. Hence, again appealing to the description of $\St_B(4)/\St_B(5)$ above, and also to Lemma~\ref{lem:relations}, we get
\begin{align*}
&\frac{\langle   [c_3,\alpha],[c_2,a,c_3],[c_1,\alpha,c_2],c_0\rangle^B}{\langle [c_3,\alpha], [c_2,a,c_3],[c_1,\alpha,c_2],[c_0,c_3^a,c_2]\rangle^B}\\
&\cong \langle c_0\rangle\times \langle c_0^a\rangle\times \langle [c_0,\alpha]\rangle\times \langle [c_0,\alpha]^a\rangle\\
&\,\times  \langle [c_0,c_3^a]\rangle\times  \langle [c_0,c_3^a]^a\rangle\times  \langle [c_0,c_3^a]^\alpha\rangle\times  \langle [c_0,c_3^a]^{\alpha a}\rangle\\
&\,\times \langle [c_0,\alpha,c_2]\rangle\times \langle [c_0,\alpha,c_2]^a\rangle\times \langle [c_0,\alpha,c_2]^\alpha\rangle\times \langle [c_0,\alpha,c_2]^{\alpha a}\rangle\\
&\,\times \langle [c_0,\alpha,c_2]^{c_3^{a\alpha}}\rangle\times \langle [c_0,\alpha,c_2]^{c_3^{a\alpha}a}\rangle\times \langle [c_0,\alpha,c_2]^{c_3^{a\alpha}\alpha}\rangle\times \langle [c_0,\alpha,c_2]^{c_3^{a\alpha}\alpha a}\rangle\\
&\,\times \frac{\langle [c_0,c_2^{\alpha},c_1^{c_3^a}]\rangle}{\langle [c_0,c_2^{\alpha},c_1^{c_3^a}]^2\rangle}\times \frac{\langle [c_0,c_2^{\alpha},c_1^{c_3^a}]^a\rangle}{\langle ([c_0,c_2^{\alpha},c_1^{c_3^a}]^a)^2\rangle}\times \frac{\langle [c_0,c_2^{\alpha},c_1^{c_3^a}]^\alpha\rangle}{\langle ([c_0,c_2^{\alpha},c_1^{c_3^a}]^\alpha)^2\rangle}\times \frac{\langle [c_0,c_2^{\alpha},c_1^{c_3^a}]^{\alpha a}\rangle}{\langle ([c_0,c_2^{\alpha},c_1^{c_3^a}]^{\alpha a})^2\rangle}\\
&\,\times \frac{\langle [c_0,c_2^{\alpha},c_1^{c_3^a}]^{c_3^a}\rangle}{\langle ([c_0,c_2^{\alpha},c_1^{c_3^a}]^{c_3^a})^2\rangle}\times \frac{\langle [c_0,c_2^{\alpha},c_1^{c_3^a}]^{c_3^a a}\rangle}{\langle ([c_0,c_2^{\alpha},c_1^{c_3^a}]^{c_3^a a})^2\rangle}\times \frac{\langle [c_0,c_2^{\alpha},c_1^{c_3^a}]^{c_3^a \alpha }\rangle}{\langle ([c_0,c_2^{\alpha},c_1^{c_3^a}]^{c_3^a \alpha })^2\rangle}\times \frac{\langle [c_0,c_2^{\alpha},c_1^{c_3^a}]^{c_3^a \alpha  a}\rangle}{\langle ([c_0,c_2^{\alpha},c_1^{c_3^a}]^{c_3^a \alpha  a})^2\rangle}\\
&\,\times \frac{\langle [c_0,c_2^{\alpha},c_1^{c_3^a}]^{c_2^\alpha}\rangle}{\langle ([c_0,c_2^{\alpha},c_1^{c_3^a}]^{c_2^\alpha})^2\rangle}\times \frac{\langle [c_0,c_2^{\alpha},c_1^{c_3^a}]^{c_2^\alpha a}\rangle}{\langle ([c_0,c_2^{\alpha},c_1^{c_3^a}]^{c_2^\alpha a})^2\rangle}\times \frac{\langle [c_0,c_2^{\alpha},c_1^{c_3^a}]^{c_2^\alpha\alpha}\rangle}{\langle ([c_0,c_2^{\alpha},c_1^{c_3^a}]^{c_2^\alpha\alpha})^2\rangle}\times \frac{\langle [c_0,c_2^{\alpha},c_1^{c_3^a}]^{c_2^\alpha\alpha a}\rangle}{\langle ([c_0,c_2^{\alpha},c_1^{c_3^a}]^{c_2^\alpha\alpha a})^2\rangle}\\
&\,\times \frac{\langle [c_0,c_2^{\alpha},c_1^{c_3^a}]^{c_2^\alpha c_3^a}\rangle}{\langle ([c_0,c_2^{\alpha},c_1^{c_3^a}]^{c_2^\alpha c_3^a})^2\rangle}\times \frac{\langle [c_0,c_2^{\alpha},c_1^{c_3^a}]^{c_2^\alpha c_3^a a}\rangle}{\langle ([c_0,c_2^{\alpha},c_1^{c_3^a}]^{c_2^\alpha c_3^a a})^2\rangle}\times \frac{\langle [c_0,c_2^{\alpha},c_1^{c_3^a}]^{c_2^\alpha c_3^a\alpha}\rangle}{\langle ([c_0,c_2^{\alpha},c_1^{c_3^a}]^{c_2^\alpha c_3^a\alpha})^2\rangle}\times \frac{\langle [c_0,c_2^{\alpha},c_1^{c_3^a}]^{c_2^\alpha c_3^a\alpha a}\rangle}{\langle ([c_0,c_2^{\alpha},c_1^{c_3^a}]^{c_2^\alpha c_3^a\alpha a})^2\rangle}\\
&\cong C_2\times \overset{32}\dots\times C_2.
\end{align*}
As before we see that the first-level projections of half of the above 32 generators yield all the generators of $\St_B(4)\backslash\St_B(5)$ in the left component with the trivial element in the right component, and the remaining 16 generators yield all the generators of $\St_B(4)\backslash\St_B(5)$ in the right component with the trivial element in the left component. Therefore we obtain that 
$\St_B(6)=\langle   [c_3,\alpha],[c_2,a,c_3],[c_1,\alpha,c_2],[c_0,c_3^a,c_2]\rangle^B$.
\end{proof}

\begin{prop}\label{prop:csp}
    Let $B$ be the second Basilica of the Grigorchuk--Erschler group. Then $B$ has the congruence subgroup property.
\end{prop}

\begin{proof} 
By \cite[Prop. 2.4]{FAGUA}, and also \cite[Lem.~3.4]{DNT}, it suffices to show that $B'' \geq \St_B(6) \times 1\times 1 \times 1$. From the previous proof, we have $\St_B(6)=\langle   [c_3,\alpha],[c_2,a,c_3],[c_1,\alpha,c_2],[c_0,c_3^a,c_2]\rangle^B$. We obtain the first two normal generators  from
    \begin{align*}
        [[c_1,a],[c_2,\alpha]] &= [(c_2,c_2),([c_3,a],1)] = [(c_3,1,c_3,1),(c_0\alpha,\alpha c_0,1,1)]\\
            & = ([c_3,c_0\alpha],1,1,1) = ([c_3,\alpha],1,1,1),    \\
            [[c_0,a],[c_3^{a},\alpha],[c_1,a]] &= [(c_1,c_1),([\alpha,a],1),(c_2,c_2)] = [(c_2,1,c_2,1),(a,a, 1,1),(c_3,1,c_3,1)]\\
            & = ([c_2,a,c_3],1,1,1).
            \end{align*}
For the next normal generator $[c_1,\alpha,c_2]$, we deduce from the Hall-Witt identity that 
\[
[c_1,\alpha,c_2]^{\alpha^{-1}}=([\alpha^{-1},c_2^{-1},c_1]^{-1})^{c_2}=([\alpha,c_2,c_1]^{-1})^{c_2}.
\]
Then from 
\begin{align*}
     [[c_2,\alpha],[c_0,a]] &= [([c_3,a],1),(c_1,c_1)] = [(c_0\alpha,\alpha c_0,1,1),(c_2,1,c_2,1)]\\
            & = ([c_0\alpha,c_2],1,1,1) = ([\alpha,c_2],1,1,1),
\end{align*}
we obtain
\[
 [[c_2,\alpha],[c_0,a],[c_3,\alpha]]^{[c_0,a]c_2^\alpha} = ([\alpha,c_2,c_1]^{c_2\alpha},1,1,1).
\]
Lastly, consider
\begin{align*}
        [[c_2,a],[c_1,a]^{[c_3^a,\alpha]}] &= [(c_0,\alpha,c_0,\alpha),(c_3,1, c_3,1)^{(a,a,1,1)}]\\
        &=[(c_0,\alpha,c_0,\alpha),(c_3^a,1, c_3,1)]\\
          &=([c_0,c_3^a],1 ,1,1),
 \end{align*}
so we obtain the final normal generator via
\[
[[c_2,a],[c_1,a]^{[c_3^a,\alpha]},[c_0,a]]= ([c_0,c_3^a,c_2],1 ,1,1).\qedhere
\]
\end{proof}

\smallskip

Next we compute the Hausdorff dimension of the closure $\overline{B}$ of $B$ in $\Aut T$. The formula for the Hausdorff dimension for Basilica groups given in \cite[Lem.~4.13 and Prop.~4.15]{PR} does not apply, since the Grigorchuk--Erschler group would have to be very close to being abelian. Instead, we use the previous result to compute the Hausdorff dimension of $B$ using \u{S}uni\'{c}'s method \cite[Prop.~6]{Sunic}.

\begin{prop}
    \label{prop:hausdorff-dimension}
   Let $B$ be the second Basilica of the Grigorchuk--Erschler group. The Hausdorff dimension of the closure $\overline{B}$ of $B$ in $\Aut T$ is $\frac{31}{32}$.
\end{prop}

\begin{proof}
    We use  \cite[Prop. 6]{Sunic} and we adopt the notation given there; that is, let $r,s,t$ be non-negative integers such that 
    \begin{enumerate}
        \item [$\bullet$] $\psi_1(\St_B(n+1))=\St_B(n)\times \St_B(n)$ for every $n\ge s$;
      \item [$\bullet$] $\log_2|G:\St_B(s)|=r$;
      \item [$\bullet$] $\log_2|G\times G:\St_B(1)|=t$.
     \end{enumerate}
     Then by \cite[Prop. 6(b)]{Sunic}, the Hausdorff dimension of $\overline{B}$ in $\Aut T$ is  $\tfrac{r-t+1}{2^s}$. 
     
     Now, it follows from the description of level stabilisers in Proposition~\ref{prop:stabilisers} that $\St_B(6) \leq B'$, and since $B$ is regular branch over $B'$, by \cite[Lem.~10]{Sunic}, we get 
    \[
    \psi_1(\St_B(n+1)) = \St_B(n) \times \St_B(n)\]
    for $n \geq 6$. So we set $s=6$ in \u{S}uni\'{c}'s formula. It follows again from Proposition \ref{prop:stabilisers} that $[B:\St_B(6)] = 2^{63}$. Thus $r = 63$ in \u{S}uni\'{c}'s formula. Now it suffices to find the index of $\St_B(1)$ in $B \times B$. We consider the quotient of $B\times B$ by $\St_B(1)$. Note that  $B \times B$ is generated by the set $\{(x,1),(1,y) \mid x,y \in \{a, \alpha, c_0,c_1,c_2,c_3\}\}$. It follows from the description of $\St_B(1)$ that  
    \[ 
    \frac{B\times B}{\St_B(1)}  = \frac{\langle (\alpha,1),(1,\alpha),(c_0,1),(1,c_0) \rangle \St_B(1)}{\St_B(1)}.
    \]
Since $c_3 = (c_0,\alpha) \in \St_B(1)$,  we have $(c_0,1) \equiv (\alpha,1) \pmod {\St_B(1)}$ and also $(1,c_0) \equiv (1,\alpha) \pmod {\St_B(1)}$. Therefore,
 \[ 
    \frac{B\times B}{\St_B(1)}  = \frac{\langle (\alpha,1),(1,\alpha) \rangle \St_B(1)}{\St_B(1)}.
    \]
We claim that $(\alpha, 1) \notin \St_B(1)$.  Indeed, if it were, then it would be in $\St_B(2)\backslash\St_B(3)$. Hence we would have the image of $(\alpha,1)$ in the quotient group $\St_B(2)/\St_B(3)$. From the proof of Proposition~\ref{prop:stabilisers}, we clearly see that this is not the case. Note further that  $(\alpha,1)$ and $(1,\alpha)$ are not equivalent modulo $\St_B(1)$. So we get that 
\[ 
    \frac{B\times B}{\St_B(1)}  \cong C_2 \times C_2.
    \]
Thus $|B \times B : \St_B(1)| = 2^2$ and so $t=2$ in \u{S}uni\'{c}'s formula.  The Hausdorff dimension of the closure $\overline{B}$ of $B$ in $\Aut T$ is given by 
\[
\frac{r-t+1}{2^s} = \frac{63-2+1}{2^6} = \frac{62}{2^6} = \frac{31}{32}.\qedhere
\]
\end{proof}

\smallskip

\subsection{A rigidity result about the automorphism group of $B$}

Recall that a group $G\le \Aut T$ is said to be \emph{saturated} if for any  $n\in\mathbb{N}$ there exists a subgroup $H_n \leq \St_G(n)$ that is characteristic in~$G$ and acts spherically transitively on every $n$th level subtree. If $G$ is saturated, then \cite[Thm.~7.5]{LN} yields that $\Aut G=N_{\Aut T}(G)$. Examples of saturated groups acting on rooted trees are, among others, the first Grigorchuk group~\cite{LN}, the second Grigorchuk group~\cite{NT}, the $p$-Basilica groups~\cite{DNT}, the branch multi-EGS groups~\cite{TUA}, and the generalised Brunner-Sidki-Vieira groups~\cite{FAGP}.

We first record some useful results.

 \begin{lemma}
     \label{lem:gamma-3}
    Let $B$ be the second Basilica of the Grigorchuk--Erschler group. Then $\gamma_3(B)\le\St_B(2)$.
 \end{lemma}
 
 \begin{proof}
     From Proposition~\ref{prop:stabilisers}, we have that $\St_B(2)=\langle c_3,c_2,c_1,c_0\rangle^B$. Hence it suffices to observe that $[\alpha,a,a]=[\alpha,a,\alpha]=1$.
 \end{proof}

\begin{theorem}\label{thm:rigidity}
The second Basilica group $B$ of the Grigorchuk--Erschler group is saturated, and hence $\Aut B=N_{\Aut T}(B)$.
\end{theorem}

\begin{proof}
We need to show that,  for any $n\in\mathbb{N}$, there exists a subgroup $H_n \leq \St_{B}(n)$ that is characteristic in~$B$ and acts spherically transitively on every subtree of the $n$th level. For $n=1$, we can take $H_1=B'$. Indeed, from $[a,\alpha]$ we have that $a\in(H_1)_j$ for $j\in\{0,1\}$. As $B$ is regular branch over $B'$, it follows that $a\in(H_1)_v$ for every vertex $v$ of level at least 1. So $H_1=B'$ acts spherically transitively on every first-level subtree.

For $H_2$, we consider $\gamma_3(B)$,  which from Lemma~\ref{lem:gamma-3} is in $\St_B(2)$. From considering $[a,\alpha,c_3^a]$ we get $a\in (\gamma_3(B))_{00}$ and likewise for all vertices at level~2. Since from the proof of Lemma~\ref{lemma branch} we see that $\gamma_3(B)\times\gamma_3(B)\le \gamma_3(B)$, it follows that $a\in(H_2)_v$ for every vertex $v$ of level at least 2.

For $n=3$, consider $H_2^2=\gamma_3(B)^2$. Note that since $a,c_1,c_2,c_3, c_0\alpha\in (\gamma_3(B))_{00}$ from considering $[c_i,a,\alpha]$ for $i\in\{0,1,2,3\}$ and 
$[a,\alpha,c_3^a]$, we have  $(c_1a)^2,(c_2a)^2,(c_3a)^2, (c_0\alpha a)^2\in (\gamma_3(B)^2)_{00}$. Thus $c_2,c_3,c_0\alpha, c_1a\in (\gamma_3(B)^2)_{000}$, which yields that $H_3:=H_2^2$ acts transitively at every subtree rooted at level 3. Note that we also appeal to the fact that $\gamma_3(B)^2\times\gamma_3(B)^2\le \gamma_3(B)^2$.

For $n>3$, we claim that we can set $H_n=H_2^{2^{n-2}}$. Indeed, it suffices to show that $(\gamma_3(B)^4)_{0000}$ contains $c_1a,c_2,c_3,c_0\alpha$, equivalently, that $c_1a,c_2,c_3,c_0\alpha\in(\langle c_2,c_3,c_0\alpha , c_1a\rangle^2)_0$. The result follows from considering $(c_1a)^2,(c_2c_1a)^2,(c_3c_1a)^2,(c_0\alpha c_1a)^2$.

The final statement then follows from~\cite[Thm.~7.5]{LN}.
\end{proof}

\section{Basilica operation preserves dense subgroups}\label{sec:operation-dense}

Here we prove Theorem~\ref{thm:Basilica-dense}, which holds more generally for groups acting on any regular rooted tree.

\begin{proof}[Proof of Theorem~\ref{thm:Basilica-dense}]
Since $G$ has the congruence subgroup property, the condition that the subgroup $H$ is dense in $G$ is equivalent to the condition $H\St_G(k)=G$ for all $k\in \mathbb{N}$. Likewise, since $B_n(G)$ has the congruence subgroup property, for $B_n(H)$ to be dense in $B_n(G)$, it suffices to show that $B_n(H)\St_{B_n(G)}(k)=B_n(G)$ for all $k\in \mathbb{N}$.
    
   This follows from the definition of the Basilica operation. Indeed, if $g$ is a generator of $G$, then $\beta_i^n(g)$ is a generator of $B_n(G)$. Let $h\in H$ be such that $h\equiv g \pmod {\St_G(k)}$. Then for all $i\in \{0,1,\ldots, n-1\}$, we have $\beta_i^n(h)\equiv \beta_i^n(g) \pmod {\St_{B_n(G)}(k)}$. Hence the result.
\end{proof}

\section{Proper dense subgroups}\label{sec:maximal}

Recall the proper dense subgroups $H(q)=\langle (ab)^q,b_0,b_1\rangle$ of $G$ that were considered in \cite{FG}; here $b=b_0b_1$ and $q$ is an odd prime. By Theorem~\ref{thm:Basilica-dense}, it follows that the subgroup $B_2(H(q))=\langle \beta_0((ab)^q), \beta_1((ab)^q), c_0,c_1,c_2,c_3\rangle$ is dense. However due to the generators of $B_2(H(q))$ being more complicated to work with, we instead consider a different dense subgroup, as seen below. In this section, we follow the strategy laid out in~\cite[Sec.~5]{FG}. We begin by recalling the following result of P.-H. Leemann.

\begin{prop}
\label{pro:4.1}\cite[Prop.~4.1]{FG}
Let $\mathcal{T}$ be the $m$-adic tree for some $m\ge 2$. Let $\mathcal{G}\le\Aut \mathcal{T}$ be countably generated by $S=\{g_1,g_2,\ldots\}$. Suppose that $m_1,m_2,\ldots\in\mathbb{N}$ are coprime with $|\mathcal{G}/\St_{\mathcal{G}}(n)|$ for all $n\in\mathbb{N}$. Then $H:=\langle g_1^{m_1},g_2^{m_2},\ldots\rangle$ is a dense subgroup of $\mathcal{G}$ with respect to the congruence topology (i.e. the topology defined via $\St_{\mathcal{G}}(n)$ for all $n\in\mathbb{N}$).
\end{prop}

Recall that  $B$ is the second Basilica of the Grigorchuk--Erschler group. Hence, setting $d:=c_0c_1c_2c_3$,
 we have that, for an odd integer $q$,
\[
\widetilde{H}(q):=\langle  \alpha, (ad)^q, c_0,c_1,c_2,c_3\rangle\le B
\]
is dense with respect to the profinite topology, since $B$ has the congruence subgroup property by Proposition~\ref{prop:csp}.

Recall that the binary rooted tree can be identified with $\{0,1\}^*$. For later use, we give  here an alternative description of the action of the generators of $B$ on $\{0,1\}^*$. Let $s\in\{0,1\}^*$ and $s_1,s_2,s_3,s_4,s_5,s_6\in\{0,1\}$. 
\begin{enumerate}
    \item [$\bullet$] The action of $a$ is given by
    \[
    a\cdot s_1s=(s_1+1)s.
    \]
    \item [$\bullet$] The action of $\alpha$ is given by
    \[
    \alpha\cdot s_1s_2s=\begin{cases}
     s_1(s_2+1)s & \text{if }s_1=0,\\
   s_1s_2s & \text{if }s_1=1.
     \end{cases}
    \]
     \item [$\bullet$] The action of each of the elements $c_3, c_2,c_1,c_0$ is given recursively by 
\begin{align*}
    c_3\cdot s_1s_2s_3s_4s
      &=\begin{cases}
   s_1s_2s_3s_4(c_3(s)) & \text{if }s_1s_2s_3s_4=0000,\\
    s_1s_2(s_3+1)s_4s & \text{if }s_1s_2=10,\\
    s_1s_2s_3s_4s & \text{otherwise;}
     \end{cases}
    \end{align*}
       
    \begin{align*}
    c_2\cdot s_1s_2s_3s_4s=\begin{cases}
    s_1s_2s_3s_4(c_2(s)) &\text{if }s_1s_2s_3s_4=0000,\\
    s_1s_2s_3(s_4+1)s & \text{if }s_1s_2s_3=010,\\
    s_1s_2s_3s_4s & \text{otherwise;}
   \end{cases}
    \end{align*}
     
    \begin{align*}
    c_1\cdot s_1s_2s_3s_4s_5s=\begin{cases}
    s_1s_2s_3s_4(c_1(s_5s)) &\text{if }s_1s_2s_3s_4=0000,\\
    s_1s_2s_3s_4(s_5+1)s & \text{if }s_1s_2s_3s_4=0010,\\
    s_1s_2s_3s_4s_5s & \text{otherwise;} 
   \end{cases}
    \end{align*}
    
    \begin{align*}
    c_0\cdot s_1s_2s_3s_4s_5s_6s=\begin{cases}
    s_1s_2s_3s_4(c_0(s_5s_6s)) &\text{if }s_1s_2s_3s_4=0000,\\
    s_1s_2s_3s_4s_5(s_6+1)s &\text{if }s_1s_2s_3s_4s_5=00010,\\
    s_1s_2s_3s_4s_5s_6s &\text{otherwise.}
   \end{cases}
    \end{align*}
\end{enumerate}

\begin{rmk}\label{rmk:5.3}
Observe that $d=(d,\alpha)$. Also note that $d(0^k 11 s)=0^k 11 s$, for any $k \in \mathbb{N}_0$ and $s\in X^*$. For all other vertices not of the form $0^k 11 s$, for any $k \in \mathbb{N}_0$ and $s\in X^*$, we see that $d\cdot s$ is the string obtained by adding 1 to the element immediately following the first $10$ in the sequence.  Therefore, in $\partial T$ the fixed boundary points of~$d$ are precisely the set of vertices $\{000\cdots\}\cup\{ 0^k11s \mid k \in \mathbb{N}_0, s \in X^*\}$; here the \textit{boundary} $\partial T$ of $T$ corresponds naturally to infinite simple rooted paths.
\end{rmk}

\begin{rmk}\label{rmk:5.4}
For $x\in\langle c_0,c_1,c_2,c_3\rangle$ and $s\in\{0,1\}^*$, either $x\cdot s=s$ or $x\cdot s=d\cdot s$. It follows that for any $\xi\in\partial T$, the orbit of $\xi$ under the action of $B$ coincides with the orbit of $\xi$ under the action of $\langle a,\alpha, d\rangle\le B$.
\end{rmk}

Now we consider the orbital graph, denoted by~$\mathfrak{O}$, generated by the action of $\langle a, \alpha, d \rangle$ on the orbit of~$\overline{0}$, where $\overline{0}$ represents the boundary point~$000\cdots$. We define the main trunk~$\mathfrak{T} = \mathfrak{T}_0 \cup \mathfrak{T}_1$ of~$\mathfrak{O}$ as the union of the orbital graphs $\mathfrak{T}_0$ and $\mathfrak{T}_1$, which are obtained by the actions of $\langle a, d \rangle$ and $\langle \alpha, d \rangle$ on the orbit of $\overline{0}$, respectively.

\begin{lemma}
    \label{lem:vertices-in-picture}
   Any $s\in\mathfrak{T}_0$ is  a sequence  with  an even number of zeros at the start, and all 1's are separated by an odd number of zeros.    Also  $\mathfrak{T}_0$ is a half-line with a loop at one end.
\end{lemma}

\begin{proof}
    We observe that the action of $a$ on $s$ swaps the digit  in the first position, i.e. at the beginning of the sequence describing a vertex. Then, as noted in Remark~\ref{rmk:5.3}, the action of $d$ on $s$ swaps the digit  immediately after the first occurrence of  the subsequence $10$, provided that there is no subsequence $11$ preceding it. The first statement then follows, as the action of an element of the form $a^\epsilon(da)^nd^\delta$, for some $n\ge 0$ and $\epsilon,\delta\in\{0,1\}$, on $\overline{0}$ involves swapping digits in an odd position  of the sequence (counting from the left).

    Then as in \cite[Prop. 5.5]{FG}, we observe that the orbital graph is a connected graph where every vertex has degree 2. This gives four possibilities for the graph: a circle, a line, a segment with loops at both ends, or a half-line with a loop at one end. It follows from Remark \ref{rmk:5.3} and the above paragraph    that $d$ has exactly one fixed point in this orbital graph, i.e. at $\overline{0}$, and so $d$ gives a loop at $\overline{0}$. Indeed, $d$ has fixed points in other orbital subgraphs, such as the vertices $0\cdots 0 11 s$, but by the first part, there are no such vertices on $\mathfrak{T}_0$. Then as $a$ has no fixed points, altogether we get a half-line with a loop at one end;  cf. Figure \ref{fig:half-line T0}.   
\end{proof}

\begin{lemma}
     \label{lem:vertices-in-picture-1}
   Any $s\in\mathfrak{T}_1$ is  a sequence  with  an odd number of zeros at the start, and all 1's are separated by an odd number of zeros. 
   Also  $\mathfrak{T}_1$ is a half-line with a loop at one end.
\end{lemma}

\begin{proof}
    This is analogous to the previous proof, since the action of $\alpha$ on $s$ swaps the digit in the second position and hence all vertices are shifted from $\mathfrak{T}_0$ by adding a zero at the front; cf. Figure \ref{fig:half-line T1}.   
\end{proof}

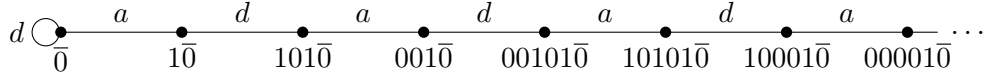
\begin{figure}
    \centering

\begin{tikzpicture}
\node at (12,0) {$\ldots$};
\draw (0,0) -- (11.6,0);
\node at (0,0) [circle,fill,inner sep=1.5pt]{};
\node [below] at (0,-0.05) {$\overline{0}$};
	
\node[circle,draw] (c) at (-0.2,0){};
\node [left] at (-0.35,0) {$d$};
\node [above] at (0.8,0) {$a$};
\node at (1.6,0) [circle,fill,inner sep=1.5pt]{};
\node [below] at (1.6,0) {$1\overline{0}$};
\node at (3.2,0) [circle,fill,inner sep=1.5pt]{};
\node [below] at (3.2,0) {$101\overline{0}$};
\node [above] at (2.4,0) {$d$};

\node at (4.8,0) [circle,fill,inner sep=1.5pt]{};
\node [below] at (4.8,0) {$001\overline{0}$};
\node [above] at (4,0) {$a$};

\node at (6.4,0) [circle,fill,inner sep=1.5pt]{};
\node [below] at (6.4,0) {$00101\overline{0}$};
\node [above] at (5.6,0) {$d$};

\node at (8,0) [circle,fill,inner sep=1.5pt]{};
\node [below] at (8,0) {$10101\overline{0}$};
\node [above] at (7.2,0) {$a$};

\node at (9.6,0) [circle,fill,inner sep=1.5pt]{};
\node [below] at (9.6,0) {$10001\overline{0}$};
\node [above] at (8.8,0) {$d$};

\node at (11.2,0) [circle,fill,inner sep=1.5pt]{};
\node [below] at (11.2,0) {$00001\overline{0}$};
\node [above] at (10.4,0) {$a$};
\end{tikzpicture}
\caption{The half-line $\mathfrak{T}_0$.}
    \label{fig:half-line T0}
\end{figure}

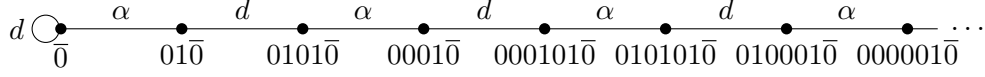
\begin{figure}
    \centering

\begin{tikzpicture}
\node at (12,0) {$\ldots$};
\draw (0,0) -- (11.6,0);
\node at (0,0) [circle,fill,inner sep=1.5pt]{};
\node [below] at (0,-0.05) {$\overline{0}$};
	
\node[circle,draw] (c) at (-0.2,0){};
\node [left] at (-0.35,0) {$d$};
\node [above] at (0.8,0) {$\alpha$};
\node at (1.6,0) [circle,fill,inner sep=1.5pt]{};
\node [below] at (1.6,0) {$01\overline{0}$};
\node at (3.2,0) [circle,fill,inner sep=1.5pt]{};
\node [below] at (3.2,0) {$0101\overline{0}$};
\node [above] at (2.4,0) {$d$};

\node at (4.8,0) [circle,fill,inner sep=1.5pt]{};
\node [below] at (4.8,0) {$0001\overline{0}$};
\node [above] at (4.0,0) {$\alpha$};

\node at (6.4,0) [circle,fill,inner sep=1.5pt]{};
\node [below] at (6.4,0) {$000101\overline{0}$};
\node [above] at (5.6,0) {$d$};

\node at (8,0) [circle,fill,inner sep=1.5pt]{};
\node [below] at (8,0) {$010101\overline{0}$};
\node [above] at (7.2,0) {$\alpha$};

\node at (9.6,0) [circle,fill,inner sep=1.5pt]{};
\node [below] at (9.6,0) {$010001\overline{0}$};
\node [above] at (8.8,0) {$d$};

\node at (11.2,0) [circle,fill,inner sep=1.5pt]{};
\node [below] at (11.2,0) {$000001\overline{0}$};
\node [above] at (10.4,0) {$\alpha$};
\end{tikzpicture}
\caption{The half-line $\mathfrak{T}_1$.}
    \label{fig:half-line T1}
\end{figure}

For odd $n\ge 3$, write $\eta(n)=\sum_{i=0}^{\frac{n-3}{2}}2^i=1+2+2^2+\cdots+2^{\frac{n-3}{2}}$, and for convenience let $\eta(1)=0$.

\begin{lemma}\label{lem:ada}
    For every $s\in X^*$, 
    \begin{enumerate}
        \item [(i)] $a(da)^{\eta(n)}(0^ns)=0^{n-1}1s$  for odd $n\ge 1$;

         \item [(ii)] $a(da)^{\eta(n)}(0^{n-1}1s)=0^ns$ for odd  $n\ge 1$.
    \end{enumerate}
\end{lemma}

\begin{proof}
We proceed by induction on~$n$, where the result is straightforward for $n=1$. We check for $n=3$. Noting that $\eta(3) = 1$, we have
\begin{align*}
    a(da)^{\eta(3)}(0^3s) = ada(0^3s) = ad(100s) = a(101s) = 001s = 0^21s,\\
    a(da)^{\eta(3)}(0^21s) = ada(0^21s) = ad(101s) = a(100s) = 000s = 0^3s.
\end{align*}
Note that $a(da)^{\eta(n)}=\big(a(da)^{\eta(n-2)}\big) d\big(a(da)^{\eta(n-2)}\big)$. By induction we have
\begin{align*}
    a(da)^{\eta(n)}(0^ns) & =\big(a(da)^{\eta(n-2)}\big) d\big(a(da)^{\eta(n-2)}\big)(0^{n-2}00s) = \big(a(da)^{\eta(n-2)}\big) d (0^{n-3}100s) \\
    & = \big(a(da)^{\eta(n-2)}\big) (0^{n-3}101s) = 0^{n-2}01s = 0^{n-1}1s,\\
    a(da)^{\eta(n)}(0^{n-1}1s) & =\big(a(da)^{\eta(n-2)}\big) d\big(a(da)^{\eta(n-2)}\big)(0^{n-2}01s) = \big(a(da)^{\eta(n-2)}\big) d (0^{n-3}101s) \\
    & = \big(a(da)^{\eta(n-2)}\big) (0^{n-3}100s) = 0^{n-2}00s = 0^{n}s.\qedhere
\end{align*}
\end{proof}

\begin{lemma}\label{lem:alpha}
    For every $s\in X^*$, 
    \begin{enumerate}
        \item [(i)]  $\alpha(d\alpha)^{\eta(n-1)}(0^{n}s)=0^{n-1}1s$ for even  $n\ge 2$;

         \item [(ii)] $\alpha(d\alpha)^{\eta(n-1)}(0^{n-1}1s)=0^{n}s$ for even  $n\ge 2$.
    \end{enumerate}
\end{lemma}
\begin{proof}
    This follows as in the previous proof, since upon removing the first digit, we reduce to the setting in the previous lemma with $a$'s replacing the $\alpha$'s.   
\end{proof}

We now show that all paths in $\mathfrak{O}$ branching off from vertices $s\ne \overline{0}$ on either $\mathfrak{T}_0$ or $\mathfrak{T}_1$ respectively, are finite.
\begin{lemma}\label{lem:finite-branches}
  \textup{(i)}  Let $s\ne \overline{0}$ be a vertex on $\mathfrak{T}_0$. Then the subgraph obtained under the action of  $\langle a, \alpha, d\rangle$ restricted to $(\mathfrak{O}\backslash \mathfrak{T}_0) \cup \{s\}$,   beginning with the path 
    
  \qquad\qquad   
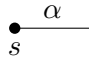
\begin{tikzpicture}
\draw (0,0) -- (1,0);
\node at (0,0) [circle,fill,inner sep=1.5pt]{};
\node [below] at (0,-0.05) {$s$};
\node [above] at (0.5,0) {$\alpha$};
\end{tikzpicture}

\noindent is finite. Further, all the vertices on this subgraph, apart from $s$, do not lie in $\mathfrak{T}_0$ or $\mathfrak{T}_1$.

\textup{(ii)}  Let $s\ne \overline{0}$ be a vertex on $\mathfrak{T}_1$. Then the subgraph obtained under the action of  $\langle a, \alpha, d\rangle$ restricted to $(\mathfrak{O}\backslash \mathfrak{T}_1) \cup \{s\}$, beginning with the path 
    
  \qquad\qquad   
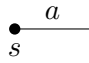
\begin{tikzpicture}
\draw (0,0) -- (1,0);
\node at (0,0) [circle,fill,inner sep=1.5pt]{};
\node [below] at (0,-0.05) {$s$};
\node [above] at (0.5,0) {$a$};
\end{tikzpicture}

\noindent is finite. Further, all the vertices on this subgraph, apart from $s$, do not lie in $\mathfrak{T}_0$ or $\mathfrak{T}_1$.
\end{lemma}

\begin{proof}
We prove both parts simultaneously. Note first that $\alpha(1s)=1s$. Hence we only need to consider vertices in $\mathfrak{T}_0\cup\mathfrak{T}_1$ starting with 0.  We proceed by induction on the number of starting zeros of the vertices. 

We begin with the two base cases of one and two starting zeros respectively. For each $s\in X^*$, the action of $\langle a, \alpha, d\rangle$ restricted to $(\mathfrak{O}\backslash \mathfrak{T}_1) \cup \{01s\}$, beginning with the path from $01s$ to $a(01s)$ yields

    
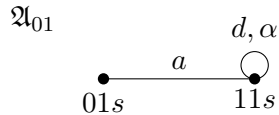
\begin{tikzpicture}
\node [left] at (-0.5,0.8) {$\mathfrak{A}_{01}$};
\draw (0,0) -- (2,0);
\node at (0,0) [circle,fill,inner sep=1.5pt]{};
\node [below] at (0,-0.05) {$01s$};
\node [above] at (1,0) {$a$};
\node at (2,0) [circle,fill,inner sep=1.5pt]{};
\node [below] at (2,0) {$11s$};
\node[circle,draw] (c) at (2,0.18){};
\node [above] at (2,0.35) {$d,\alpha$};
\end{tikzpicture}

\noindent where we denote this type of finite subgraph by $\mathfrak{A}_{01}$. Likewise, we see that  for every $s\in X^*$, the action of $\langle a, \alpha, d\rangle$ restricted to $(\mathfrak{O}\backslash \mathfrak{T}_0) \cup \{001s\}$, beginning with the path from $001s$ to $\alpha(001s)$ gives 

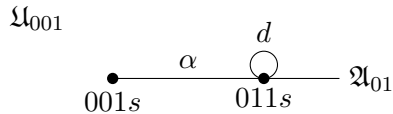
\nopagebreak
    \begin{tikzpicture}
\node [left] at (-0.5,0.8) {$\mathfrak{U}_{001}$};
\draw (0,0) -- (3,0);
\node at (0,0) [circle,fill,inner sep=1.5pt]{};
\node [below] at (0,-0.05) {$001s$};
\node [above] at (1,0) {$\alpha$};
\node at (2,0) [circle,fill,inner sep=1.5pt]{};
\node [below] at (2,0) {$011s$};
\node[circle,draw] (c) at (2,0.18){};
\node [above] at (2,0.35) {$d$};
\node [right] at (3,0) {$\mathfrak{A}_{01}$};
\end{tikzpicture}

\noindent where we denote this type of finite subgraph by $\mathfrak{U}_{001}$. It is clear that all vertices, apart from $s$, are not in $\mathfrak{T}_0$ or  $\mathfrak{T}_1$. Similarly for the vertices on $\mathfrak{A}_{01}$. Note that more generally, even if $s\notin \mathfrak{T}_0\cup \mathfrak{T}_1$, we still have that the corresponding vertices of $\mathfrak{A}_{01}$ and $\mathfrak{U}_{001}$ are not in $\mathfrak{T}_0\cup \mathfrak{T}_1$.

More generally, for some $k\in\mathbb{N}$ and for each $s\in X^*$, we denote by $\mathfrak{A}_{0^{2k-1}1}$  the type of subgraph obtained under the action of $\langle a, \alpha, d\rangle$, beginning with the path from $0^{2k-1}1s$ to  $a(0^{2k-1}1s)$. Similarly, for some $k\in\mathbb{N}$ and for each $s\in X^*$, we denote by $\mathfrak{U}_{0^{2k}1}$  the type of subgraph obtained under the action of $\langle a, \alpha, d\rangle$, beginning with the path from $0^{2k}1s$ to  $\alpha(0^{2k}1s)$.

Now by induction, we may suppose that for some $n\in\mathbb{N}$,   the subgraphs $\mathfrak{U}_{0^{2k}1}$ and $\mathfrak{A}_{0^{2k-1}1}$ are finite for all $k\in\{1,\ldots,n\}$. We now show that the subgraphs $\mathfrak{U}_{0^{2n+2}1}$ and $\mathfrak{A}_{0^{2n+1}1}$ are finite.  

We first consider $\mathfrak{A}_{0^{2n+1}1}$. From Lemma \ref{lem:ada}, we have $a(da)^{\eta(2n+1)}(0^{2n+1}1s)=0^{2n}11s$. So the ``main'' path of the subgraph  $\mathfrak{A}_{0^{2n+1}1}$ given by the action of $\langle a,d\rangle$ terminates in the finite subgraph $\mathfrak{U}_{0^{2n}1}$; see for instance Figures \ref{fig:3-0} and \ref{fig:5-0}. Indeed, recall that the action of $d$ is trivial on the vertex  $0^{2n}11s$. Therefore, it remains to consider the action of $\langle a,\alpha, d\rangle$ on the interior vertices of the ``main'' path, starting with the action of $\alpha$. Note that by first acting with $\alpha$, we depart from the subgraph obtained under the action of just $\langle a,d\rangle$. Recall that we only need to consider vertices starting with a zero, as otherwise the subgraph obtained is just a loop given by $\alpha$; indeed, we see for instance in Figures \ref{fig:3-0} and \ref{fig:5-0} that the action of $\langle a, d \rangle$ on the  ``main'' path, starting from $0^{2n+1}1s$, involves swapping digits in a certain odd position of the sequence. Hence the starting number of zeros in these remaining vertices is even  and as this number is strictly less than $2n+1$, the result follows by induction, more specifically, the side subgraphs will be one of the $\mathfrak{U}_{0^{2k}1}$ for $k\in\{1,\ldots,n\}$.

We now consider $\mathfrak{U}_{0^{2n+2}1}$. By Lemma~\ref{lem:alpha}, we have that $\alpha(d\alpha)^{\eta(2n+1)}(0^{2n+2}1s)=0^{2n+1}11s$. Since $d$ acts trivially on $0^{2n+1}11s$, the action of $\langle \alpha, d\rangle$ on $0^{2n+2}1s$ leads to a path terminating with the subgraph $\mathfrak{A}_{0^{2n+1}1}$; see for example Figure~\ref{fig:4-0}. Further note that the action of $\langle \alpha, d\rangle$ on the ``main'' path, starting from $0^{2n+2}1s$, can only swap a digit at an even position of the sequence, hence leaves an odd number of zeros before the first occurrence of 1 in each interior vertex between $0^{2n+2}1s$ and $0^{2n+1}11s$. We depart from the ``main''  path obtained by the action of $\langle \alpha, d\rangle$ only when we act on the interior vertices by $a$. This leads to subgraphs of type $\mathfrak{A}_{0^{2k+1}1}$, for $k \in \{1,\dots,n\}$, as seen for instance in Figure~\ref{fig:4-0}. Hence we are done as before by induction.

Finally, for $s\in\mathfrak{T}_0$, we observe that, apart from $s$ itself, the vertices in the orbit under the action of $\langle \alpha,d\rangle$ do not lie in $\mathfrak{T}_0\cup\mathfrak{T}_1$; cf. Lemmas~\ref{lem:vertices-in-picture} and \ref{lem:vertices-in-picture-1}. Likewise for the orbit of $s\in \mathfrak{T}_1$ under the action of $\langle a,d\rangle$.
    \end{proof}

    \begin{figure}
    \centering
    
    \begin{tikzpicture}
\draw (0,0) -- (6.5,0);
\node at (0,0) [circle,fill,inner sep=1.5pt]{};
\node [below] at (0,-0.05) {$0001s$};
\node [above] at (1,0) {$a$};
\node at (2,0) [circle,fill,inner sep=1.5pt]{};
\node [below] at (2,0) {$1001s$};
\node[circle,draw] (c) at (2,0.18){};
\node [above] at (2,0.35) {$\alpha$};
\node at (4,0) [circle,fill,inner sep=1.5pt]{};
\node [below] at (4,0) {$1011s$};
\node[circle,draw] (c) at (4,0.18){};
\node [above] at (4,0.35) {$\alpha$};
\node [above] at (3,0) {$d$};
\node [above] at (5,0) {$a$};
\node at (6,0) [circle,fill,inner sep=1.5pt]{};
\node [below] at (6,0) {$0011s$};
\node[circle,draw] (c) at (6,0.18){};
\node [above] at (6,0.35) {$d$};
\node [right] at (6.5,0) {$\mathfrak{U}_{001}$};
\end{tikzpicture}
 \caption{The finite subgraph of type $\mathfrak{A}_{0001}$}
 \label{fig:3-0}
\end{figure}

 \begin{figure}
     \centering
 \!\!\!\!\!\!\!\!\begin{tikzpicture}
\draw (0,0) -- (6.5,0);
\node at (0,0) [circle,fill,inner sep=1.5pt]{};
\node [below] at (0,-0.05) {$00001s$};
\node [above] at (1,0) {$\alpha$};
\node at (2,0) [circle,fill,inner sep=1.5pt]{};
\node [below] at (2,0) {$01001s$};
\draw (2,0) -- (2,0.25);
\node [above] at (2,0.25) {$\mathfrak{A}_{01}$};
\node at (4,0) [circle,fill,inner sep=1.5pt]{};
\node [below] at (4,0) {$01011s$};

\node [above] at (3,0) {$d$};
\draw (4,0) -- (4,0.25);
\node [above] at (4,0.25) {$\mathfrak{A}_{01}$};
\node [above] at (5,0) {$\alpha$};
\node at (6,0) [circle,fill,inner sep=1.5pt]{};
\node [below] at (6,0) {$00011s$};
\node[circle,draw] (c) at (6,0.18){};
\node [above] at (6,0.35) {$d$};
\node [right] at (6.5,0) {$\mathfrak{A}_{0001}$};
\end{tikzpicture}
 \caption{The finite subgraph of type $\mathfrak{U}_{00001}$}
 \label{fig:4-0}
\end{figure}   

\bigskip

\begin{figure}
     \centering
\!\!\!\!\!\!\!\!\!\!\!\!\!\!\!\!\!\!\!\!\!\!\!\!\begin{tikzpicture}
\draw (1,0) -- (13.5,0);
\node at (1,0) [circle,fill,inner sep=1.5pt]{};
\node [below] at (1,-0.05) {$0^51s$};
\node [above] at (1.5,0) {$a$};
\node at (2.3,0) [circle,fill,inner sep=1.5pt]{};
\node [below] at (2.3,0) {$10^31s$};
\node[circle,draw] (c) at (2.3,0.18){};
\node [above] at (2.3,0.35) {$\alpha$};
\node at (3.8,0) [circle,fill,inner sep=1.5pt]{};
\node [below] at (3.8,0) {$10101s$};
\node [above] at (3,0) {$d$};
\node[circle,draw] (c) at (3.8,0.18){};
\node [above] at (3.8,0.35) {$\alpha$};
\node [above] at (4.8,0) {$a$};
\node at (5.6,0) [circle,fill,inner sep=1.5pt]{};
\node [below] at (5.6,0) {$0^210^21s$};
\draw (5.6,0) -- (5.6,0.25);
\node [above] at (5.6,0.25) {$\mathfrak{U}_{001}$};
\node [above] at (6.5,0) {$d$};
\node at (7.5,0) [circle,fill,inner sep=1.5pt]{};
\node [below] at (7.5,0) {$0^21011s$};
\draw (7.5,0) -- (7.5,0.25);
\node [above] at (7.5,0.25) {$\mathfrak{U}_{001}$};
\node [above] at (8.4,0) {$a$};
\node at (9,0) [circle,fill,inner sep=1.5pt]{};
\node [below] at (9,0) {$101011s$};
\node[circle,draw] (c) at (9,0.18){};
\node [above] at (9,0.35) {$\alpha$};
\node [above] at (10,0) {$d$};
\node at (11,0) [circle,fill,inner sep=1.5pt]{};
\node [below] at (11,0) {$10^311s$};
\node[circle,draw] (c) at (11,0.18){};
\node [above] at (11,0.35) {$\alpha$};
\node [above] at (12,0) {$a$};
\node at (13,0) [circle,fill,inner sep=1.5pt]{};
\node [below] at (13,0) {$0^411s$};
\node[circle,draw] (c) at (13,0.18){};
\node [above] at (13,0.35) {$d$};
\node [right] at (13.5,0) {$\mathfrak{U}_{00001}$};
\end{tikzpicture}
 \caption{The finite subgraph of type $\mathfrak{A}_{000001}$}
 \label{fig:5-0}
 \end{figure}

Recall that $\mathfrak{O}$ is the orbital graph obtained by the action of $\langle a, \alpha, d \rangle$ on the orbit of~$\overline{0}$ and $\mathfrak{T} = \mathfrak{T}_0 \cup \mathfrak{T}_1$ is the main trunk of $\mathfrak{O}$. Denote by $\mathfrak{G}$ the difference $\mathfrak{O} \backslash \mathfrak{T}$.

\begin{prop}\label{prop:maximal-trunk}
    The  difference $\mathfrak{G}$, of the orbital graph~$\mathfrak{O}$ minus the main trunk~$\mathfrak{T}$, is a forest where each tree is finite.
\end{prop}

\begin{proof}
    This follows from the previous lemma.
\end{proof}

Next, we consider the half-line~$\mathfrak{T}_0$ of~$\mathfrak{T}$ which is induced by just the action of   $\langle a,d\rangle$ on the orbit of $\overline{0}$.
\begin{prop}
\label{pro:5.6}
There is a bijection $\zeta:\mathbb{Z}\rightarrow\mathfrak{T}_0$ given by $\zeta(n)=(ad)^n\cdot \overline{0}$.
\end{prop}

\begin{proof}
Surjectivity follows from the definition of $\mathfrak{T}_0$. Indeed, let $s\in\mathfrak{T}_0$ be a vertex on the half-line. If,  in the  graph $\mathfrak{T}_0$, the distance  from $\overline{0}$ to  $s$  is an even number, say $2n$ for some $n\ge 0$, then $s=(da)^n\cdot \overline{0}=(ad)^{-n}\cdot \overline{0}$; cf. Figure~\ref{fig:half-line T0}. Similarly, if the distance  from $\overline{0}$ to  $s$  is an odd number, say $2n-1$ for some $n\ge 1$, then $s=(ad)^n\cdot \overline{0}$.

Appealing to Figure~\ref{fig:half-line T0} again, one sees that distinct integers $n$ and $m$ correspond to distinct vertices $(ad)^n\cdot \overline{0}$  and $(ad)^m\cdot \overline{0}$. So $\zeta$ is injective as well.
\end{proof}

\begin{theorem}
Let  $B$ be the second Basilica of the Grigorchuk--Erschler group. Then for each odd $q\in\mathbb{N}$, the subgroup $\widetilde{H}(q)=\langle \alpha, (ad)^q, c_0,c_1,c_2,c_3\rangle$ is proper and dense with respect to the profinite topology.
\end{theorem}

\begin{proof}
The fact that $\widetilde{H}(q)$ is dense was established before, so we just need to show that it is proper.  By Remark~\ref{rmk:5.4}, the orbit of $\overline{0}$ under the action of $\widetilde{H}(q)$ on $\mathfrak{T}_0$ is the same as the orbit of $\overline{0}$ under the action of $\langle \alpha, (ad)^q,d\rangle$. Similar to \cite[Prop.~5.6]{FG}, we consider the bijection $\zeta:\mathbb{Z}\rightarrow \mathfrak{T}_0$. As in \cite[Proof of Thm.~5.1]{FG}, the bijection $\zeta$ allows us to define an action of $\langle a,d\rangle$ on $\mathbb{Z}$ via
\[
g\cdot n:=\zeta^{-1}(g\cdot \zeta(n)),
\]
where $g\in \langle a,d\rangle$ and $n\in\mathbb{Z}$. Then we observe  that the orbit of $\overline{0}$ under the action of $\langle  (ad)^q, d\rangle$ corresponds to $q\mathbb{Z}$ under $\zeta$. In light of Lemma~\ref{lem:finite-branches}(i), the action of $\alpha$  on $\mathfrak{T}_0$ can be ignored; indeed, including the action of $\alpha$ on the vertices of $\mathfrak{T}_0$ either results in loops, or finite subgraphs departing from $\mathfrak{T}_0$, or we land in $\mathfrak{T}_1$ at $01\overline{0}$. Recall that the orbit of $\overline{0}$ under the action of $\widetilde{H}(q)$ on $\mathfrak{T}_0$ is the same as the orbit of $\overline{0}$ under the action of $\langle \alpha, (ad)^q,d\rangle$.  As the orbit of $\overline{0}$ under the action of  $\langle  (ad)^q,d\rangle$ corresponds to $q\mathbb{Z}$ under $\zeta$, and as noted above, the action of $\langle \alpha, (ad)^q,d\rangle$ on the orbit of $\overline{0}$ restricted to $\mathfrak{T}_0$ is the same as the action of $\langle  (ad)^q,d\rangle$,
it then follows that $\widetilde{H}(q)$ is a proper subgroup of $B$.
\end{proof}

\begin{coro}
Let  $B$ be the second Basilica of the Grigorchuk--Erschler group. Then $B$ has at least countably infinitely many distinct maximal subgroups of infinite index.
\end{coro}

\begin{proof}
This follows from the previous theorem, noting that if two odd  $q_1,q_2\in\mathbb{N}$ are coprime, then $\widetilde{H}(q_1)$ and $\widetilde{H}(q_2)$ are contained in different maximal subgroups of infinite index.
\end{proof}

\end{document}